\documentclass[10pt,a4paper,notitlepage, oneside]{article}
\usepackage[utf8x]{inputenc} % umlaute und so
\usepackage{ucs} %umlaute, unterstützt package utf8x
\usepackage[english]{babel}
\usepackage{amsmath}
\usepackage{amsthm}
\usepackage{amsfonts}
\usepackage{amssymb}
\usepackage{mathtools}
\usepackage{cite}
\usepackage{hyperref}

\usepackage{enumerate} 

\usepackage{tikz}
\usetikzlibrary{backgrounds}

\usepackage{todonotes}  % put the option disable to the package to remove displayed todo notes.

\newtheorem{definition}{Definition}

\newtheorem{theorem}[definition]{Theorem}

\newtheorem{observation}[definition]{Observation}

\newtheorem{lemma}[definition]{Lemma}

\newcommand{\ssp}{\textrm{\rm SSP}}
\newcommand{\tstab}{\textrm{\rm TSTAB}}
\newcommand{\charf}[1]{\raisebox{\depth}{\(\chi\)}_{#1}}

\DeclareMathOperator{\dist}{dist}

\newcommand{\Z}{\ensuremath{\mathbb{Z}}} % integer numbers
\newcommand{\R}{\ensuremath{\mathbb{R}}} % real numbers

\usepackage{tikz}
\usetikzlibrary{arrows,positioning,intersections}
\usetikzlibrary{backgrounds}
\usetikzlibrary{trees}

\tikzstyle{hvertex}=[circle,inner sep=0.cm, minimum size=1mm, fill=white, draw=black]
\tikzstyle{novertex}=[circle,inner sep=0.cm, fill=white, draw=white]
\tikzstyle{hedge}=[thick]
\tikzstyle{contractedge}=[->,color=gray]

\title{Reducing quadrangulations of the sphere and the projective plane}
\author{Elke Fuchs and Laura Gellert}
\date{}
\begin{document}

\maketitle

\begin{abstract}
We show that every quadrangulation of the sphere
can be transformed into a $4$-cycle by deletions of degree-$2$ vertices and by  $t$-contractions at degree-$3$ vertices.  A $t$-contraction simultaneously contracts all incident edges at a vertex with stable neighbourhood. The operation is mainly used in the field of $t$-perfect graphs. \\
We further show that a non-bipartite quadrangulation of the projective plane can be transformed into an odd wheel by $t$-contractions and deletions of degree-$2$ vertices. \\
We deduce that a quadrangulation of the projective plane is (strongly) $t$-perfect if and only if the graph is bipartite.
\end{abstract}

\section{Introduction}

%Tutte~\cite{Tutte61} gave a method for generating all $3$-connected planar graphs. 
%This class of graphs is closely related to the class of $3$-connected quadrangulations without separating $4$-cycles (see Brinkmann et al~\cite{BGGMTW05}). Thus, Tutte implicitely gave a first method to generate a restricted class of quadrangulations.  

For characterising quadrangulations of the sphere, it is very useful to transform a quadrangulation into a slightly smaller one. Such reductions are mainly based on the following idea: Given a class of quadrangulations, a sequence of particular face-contractions transforms every member of the class into a $4$-cycle; see eg   Brinkmann et~al~\cite{BGGMTW05}, Nakamoto~\cite{Nakamoto99}, Negami and Nakamoto~\cite{NeNa93}, and Broersma et al.~\cite{BDG93}. 
A \emph{face-contraction} identifies two non-adjacent vertices $v_1, v_3$ of a $4$-face $v_1, v_2,v_3,v_4$ in which the common neighbours of $v_1$ and $v_3$ are only $v_2$ and $v_4$. 
%Afterwards, multiple edges are deleted. 
%A \emph{hexagonal contraction} at a vertex $x$ with neighbourhood $\lbrace  a_1, a_2, a_3 \rbrace$ deletes the edge $xa_1$ and contracts the edges $xa_2$ and $xa_3$. 
A somewhat different approach was made by Bau et al.~\cite{BMNZ14}. They showed that any quadrangulation of the sphere can be transformed into a $4$-cycle by a sequence of  deletions of degree-$2$ vertices and so called hexagonal contractions.
The obtained graph is a minor of the previous graph. Both operations can be obtained from face-contractions. 

\medskip

We provide a new way to reduce arbitrary quadrangulations of the sphere to a $4$-cycle. Our operations are minor-operations --- in contrast to face-contractions.
We use deletions of degree-$2$ vertices and $t$-contractions. A \emph{$t$-contraction} simultaneously contracts all incident edges of a vertex with stable neighbourhood and deletes all multiple edges. The operation is mainly used in the field of $t$-perfection.  Face-contractions cannot be obtained from $t$-contractions. 
We restrict ourselves to $t$-contractions at vertices that are only contained in $4$-cycles whose interior does not contain a vertex. 
\begin{gather}
\mbox{These $t$-contractions and deletions of degree-$2$ vertices  } \nonumber \sloppy \\
\mbox{ can be obtained from a sequence of face-contractions.} \label{eq:operations} \sloppy
\end{gather}
 Figure~\ref{fig:facecon} illustrates this. The restriction on the applicable $t$-contractions makes sure that all face-contractions can be applied, ie that all identified vertices are non-adjacent and have no common neighbours besides the two other vertices of their $4$-face.

\begin{figure}[bht]
\begin{center}
\begin{tikzpicture}[scale = 1]
%degree 2
\begin{scope}
\node[hvertex] (v) at (0,0){};
\node[hvertex] (x1) at (0,1){};
\node[hvertex] (x2) at (0,-1){};
\node[hvertex] (w) at (1,0){};
\node[hvertex] (w2) at (-1,0){};
\node[novertex] (l1) at (-0.5,0){\textcolor{gray}{\small{1}}};

%\node[novertex] (lv) at (-0.5,0){$v$};

\draw[hedge] (v) -- (x1);
\draw[hedge] (w) -- (x1);
\draw[hedge] (v) -- (x2);
\draw[hedge] (w) -- (x2);
\draw[hedge] (w2) -- (x2);
\draw[hedge] (w2) -- (x1);
\draw[contractedge] (v) -- (l1);
\draw[contractedge] (w2) -- (l1);
\end{scope}

%t-contraction degree 3
\begin{scope}[shift={(3,0)}]
\node[hvertex] (v) at (0,0){};
%\node[novertex] (lv) at (-0.2,0.05){$v$};
\node[hvertex] (c1) at (0,1){};
\node[hvertex] (c2) at (1,0.5){};
\node[hvertex] (c3) at (1,-0.5){};
\node[hvertex] (c4) at (0,-1){};
\node[hvertex] (c5) at (-1,-0.5){};
\node[hvertex] (c6) at (-1,0.5){};

\draw[hedge] (c1) -- (v);
\draw[hedge] (c3) -- (v);
\draw[hedge] (c5) -- (v);
\draw[hedge] (c1) -- (c2);
\draw[hedge] (c2) -- (c3);
\draw[hedge] (c3) -- (c4);
\draw[hedge] (c4) -- (c5);
\draw[hedge] (c5) -- (c6);
\draw[hedge] (c1) -- (c6);

\node[novertex] (l3) at (-0.5,0.25){\textcolor{gray}{\small{3}}};
\node[novertex] (l1) at (0.5,0.25){\textcolor{gray}{\small{1}}};
\node[novertex] (l2) at (0,-0.5){\textcolor{gray}{\small{2}}};
\draw[contractedge] (v) -- (l1);
\draw[contractedge] (c2) -- (l1);
\draw[contractedge] (c5) -- (l2);
\draw[contractedge] (c3) -- (l2);
\draw[contractedge] (c5) -- (l3);
\draw[contractedge] (c1) -- (l3);
\end{scope}

%t-contraction at degree 6
\begin{scope}[shift={(6,0)}, scale = 1]
\def\krad{1cm}
\def\rad{.6cm}

\def\angle{360/12}
\pgfmathtruncatemacro{\nminusone}{12-1}
\node[hvertex] (c) at (0,0){};
\foreach \i in {0,...,\nminusone}{
  \begin{scope}[on background layer]
    \draw[hedge] (270-\angle/2+\i*\angle:\krad) -- (270+\angle/2+\i*\angle:\krad);
  \end{scope}
  \node[hvertex] (v\i) at (270-\angle/2+\i*\angle:\krad){};
  }
\draw[hedge] (c) -- (v0);
\draw[hedge] (c) -- (v2);
\draw[hedge] (c) -- (v4);
\draw[hedge] (c) -- (v6);
\draw[hedge] (c) -- (v8);
\draw[hedge] (c) -- (v10);

\def\angle{360/6}
\foreach \i in {1,...,6}{
  \node[novertex] (l\i) at (255-\angle/2+\i*\angle:\rad){\textcolor{gray}{\small{\i}}};
  }
  
\draw[contractedge] (c) -- (l1);
\draw[contractedge] (v1) -- (l1);
\draw[contractedge] (v2) -- (l2);
\draw[contractedge] (v4) -- (l2);
\draw[contractedge] (v4) -- (l3);
\draw[contractedge] (v6) -- (l3);
\draw[contractedge] (v6) -- (l4);
\draw[contractedge] (v8) -- (l4);
\draw[contractedge] (v8) -- (l5);
\draw[contractedge] (v10) -- (l5);
\draw[contractedge] (v10) -- (l6);
\draw[contractedge] (v0) -- (l6);

\end{scope}

\end{tikzpicture}
\end{center}

\caption{Face-contractions that give a deletion of a degree-$2$ vertex, a $t$-contraction at a degree-$3$ and a degree-$6$ vertex}
\label{fig:facecon}
\end{figure}
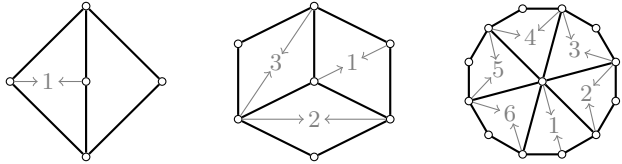

We prove:
\begin{theorem}\label{thm:plane_C4_irreducible}
Let $G$ be a quadrangulation of the sphere. Then, there is a sequence of $t$-contractions at degree-$3$ vertices and deletions of degree-$2$ vertices that transforms $G$ into a $4$-cycle. During the whole process, the graph remains a quadrangulation. 
\end{theorem}

The proof of Theorem~\ref{thm:plane_C4_irreducible} can be found in Section~\ref{sec:quadrangulations}.
It is easy to see that both operations used in Theorem~\ref{thm:plane_C4_irreducible} are necessary.
By \eqref{eq:operations}, Theorem~\ref{thm:plane_C4_irreducible} implies: 
\begin{gather*}
\mbox{Any quadrangulation of the sphere can be transformed }  \sloppy \\ 
\mbox{into  a $4$-cycle by a sequence of face-contractions.} \sloppy
\end{gather*}

Via the dual graph, quadrangulations of the sphere are in one-to-one correspondence with planar $4$-regular (not necessarily simple) graphs. Theorem~\ref{thm:plane_C4_irreducible} thus implies a method to reduce all $4$-regular planar graphs to the graph on two vertices and four parallel edges. 
Broersma et al.~\cite{BDG93}, Lehel~\cite{Le81}, and Manca~\cite{Ma79} analysed methods to reduce $4$-regular planar graphs to the octahedron graph. 

\medskip

In the second part of this paper, we consider quadrangulations of the projective plane. 
We use Theorem~\ref{thm:plane_C4_irreducible} to reduce all non-bipartite quadrangulations of the projective plane to an odd wheel. 
A \emph{$p$-wheel} $W_p$ is a graph consisting of a cycle $(w_1, \ldots, w_p,w_1)$ and a vertex $v$ adjacent to all vertices of the cycle. A wheel $W_p$ is an \emph{odd wheel}, if $p$ is odd. Figure~\ref{fig:wheels} shows some odd wheels.

\begin{theorem}\label{thm:pp_odd_wheels_irreducible}
Let $G$ be a non-bipartite quadrangulation of the projective plane. Then, there is a sequence of $t$-contractions and deletions of degree-$2$ vertices that transforms $G$ into an odd wheel. During the whole process, the graph remains a non-bipartite quadrangulation.
\end{theorem}

The proof of this theorem can be found in Section~\ref{sec:quadrangulations}. It is easy to see that both operations used in this theorem are necessary.

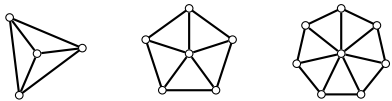
\begin{figure}[htb]
\def\krad{0.6cm}
\newcommand{\wheel}[1]{
\def\angle{360/#1}
\pgfmathtruncatemacro{\nminusone}{#1-1}
\node[hvertex] (c) at (0,0){};
\foreach \i in {0,...,\nminusone}{
  \begin{scope}[on background layer]
    \draw[hedge] (270-\angle/2+\i*\angle:\krad) -- (270+\angle/2+\i*\angle:\krad);
  \end{scope}
  \node[hvertex] (v\i) at (270-\angle/2+\i*\angle:\krad){};
  \draw[hedge] (c) -- (v\i);
}
}

\begin{center}
\begin{tikzpicture}[scale=1]

\begin{scope}[shift={(0,0)}, rotate=37]
\wheel{3}
\end{scope}
\begin{scope}[shift={(2,0)}]
\wheel{5}
\end{scope}
\begin{scope}[shift={(4,0)}]
\wheel{7}
\end{scope}

\end{tikzpicture}
\end{center}
\caption{The odd wheels $W_3, W_5$ and $W_7$}
\label{fig:wheels}

\end{figure}

\begin{figure}[bht]
\begin{center}
\begin{tikzpicture}[scale = .55]
\begin{scope} [scale = .8]
\draw[dotted] (0,0.5) circle (2.5cm);
\node[hvertex] (cr) at (2.5,0.5){};
\node[hvertex] (cl) at (-2.5,0.5){};
\node[novertex] (ur) at (3,0.5){$v$};
\node[novertex] (ul) at (-3,0.5){$v$};
\def\dist{0.2}
\node[hvertex] (p6) at (0,3){};
%\node[novertex] (n6) at (-1,-2){$w_5$};
\foreach \i in {1,2,3,4,5}{
    \node[hvertex] (p\i) at (0,-3+\i){};
%    \node[novertex](n\i) at (-1,-2+\i) {$w_{\i}$};
}

\draw[hedge] (p1) -- (p2);
\draw[hedge] (p2) -- (p3);
\draw[hedge] (p3) -- (p4);
\draw[hedge] (p4) -- (p5);
\draw[hedge] (p5) -- (p6);

\draw[hedge] (cl) -- (p2);
\draw[hedge] (cl) -- (p4);
\draw[hedge] (cl) -- (p6);

\draw[hedge] (cr) -- (p1);
\draw[hedge] (cr) -- (p3);
\draw[hedge] (cr) -- (p5);
\end{scope}

%wheel
\begin{scope} [shift={(5.5,0)}, scale = .8]
\draw[dotted] (0,0.5) circle (2.5cm);
\node[hvertex] (cr) at (2.5,0.5){};
\node[hvertex] (cl) at (-2.5,0.5){};
%\node[novertex] (ur) at (3,0.5){$v$};
%\node[novertex] (ul) at (-3,0.5){$v$};
\def\dist{0.2}
\node[hvertex] (p6) at (0,3){};
%\node[novertex] (n6) at (-1,-2){$w_5$};
\foreach \i in {1,2,3,4,5}{
    \node[hvertex] (p\i) at (0,-3+\i){};
%    \node[novertex](n\i) at (-1,-2+\i) {$w_{\i}$};
}

\draw[hedge] (p1) -- (p2);
\draw[hedge] (p2) -- (p3);
\draw[hedge] (p3) -- (p4);
\draw[hedge] (p4) -- (p5);
\draw[hedge] (p5) -- (p6);

\draw[hedge] (cl) -- (p2);
\draw[hedge] (cl) -- (p4);
\draw[hedge] (cl) -- (p6);

\draw[hedge] (cr) -- (p1);
\draw[hedge] (cr) -- (p3);
\draw[hedge] (cr) -- (p5);

\node[novertex] (l1) at (.7,1){\textcolor{gray}{\small{1}}};
\draw[contractedge] (p3) -- (l1);
\draw[contractedge] (p5) -- (l1);

\node[novertex] (l2) at (-.7,0){\textcolor{gray}{\small{2}}};
\draw[contractedge] (p4) -- (l2);
\draw[contractedge] (p2) -- (l2);

\end{scope}

\end{tikzpicture}
\end{center}

\caption{An even embedding of $W_5$ in the projective plane and face-contractions that produce a smaller odd wheel. Opposite points on the dotted cycle are identified. }
\label{fig:pp_wheels}
\end{figure}
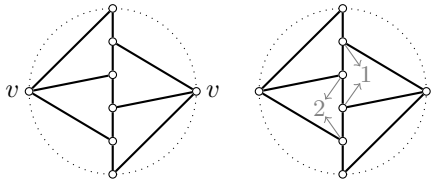

Negami and Nakamoto~\cite{NeNa93} showed that 
any non-bipartite quadrangulation of the projective plane can be transformed
into a $K_4$ by a sequence of face-contractions.
This result can be deduced from Theorem~\ref{thm:pp_odd_wheels_irreducible}:
By \eqref{eq:operations}, Theorem~\ref{thm:pp_odd_wheels_irreducible} implies that 
any non-bipartite quadrangulation of the projective plane can be transformed into an odd wheel by a sequence of face-contractions.
The odd wheel $W_{2k+1}$ can now be transformed into $W_{2k-1}$ --- and finally into $W_3=K_4$ --- by face-contractions (see Figure~\ref{fig:pp_wheels}).

Nakamoto~\cite{Nakamoto99} gave a reduction method based on face-contractions and so called $4$-cycle deletions for non-bipartite quadrangulations of the projective plane with minimum degree $3$.
%showed that face-contractions and deletions of $4$-cycles transform any non-bipartite quadrangulation of the projective plane with minimum degree $3$ into an odd wheel.
Matsumoto et al.~\cite{MNY16} analysed quadrangulations of the projective plane with respect to hexagonal contractions while Nakamoto considered face-contractions for quadrangulations of the Klein~bottle~\cite{Nakamoto95} and the torus~\cite{Nakamoto96}. 
Youngs~\cite{You96} 
%showed that all non-bipartite quadrangulations of the projective plane have chromatic number equal to $4$.
and Esperet and Stehl{\'{\i}}k~\cite{Es_Steh15} 
%gave bounds for edge- and face-width of non-bipartite quadrangulations.
considered non-bipartite quadrangulations of the projective plane with regard to vertex-colourings and width-parameters. 

%It was shown by Youngs [25] that all non-bipartite quadrangulations of the projective plane
%such graphs are 4-chromatic.  Gimbel and Thomassen [10] later proved that
%triangle-free projective-planar graphs are 3-colorable if and only if they do
%not  contain  a  non-bipartite  projective  quadrangulation,  and  used  this  to
%show that the 3-colorability of triangle-free projective-planar graphs can be
%decided in polynomial time.  Thomassen [24] also used projective quadran-
%gulations  to  give  negative  answers  to  two  questions  of  Bollob ́as  [2]  about
%4-chromatic graphs.

%An embedding of a graph into the projective plane can be constructed in linear time with an algorithm of Mohar~\cite{Mohar93}.

\medskip

Theorem~\ref{thm:pp_odd_wheels_irreducible} allows an application to the theory of $t$-perfection.
A graph $G$ is \emph{$t$-perfect} if its stable set polytope $\ssp(G)$ equals the polyhedron $\tstab(G)$. The  \emph{stable set polytope} $\ssp(G)$ is the convex hull of stable sets of $G$; the polyhedron \emph{$\tstab(G)$} is defined via non-negativity-, edge- and odd-cycle inequalities (see Section~\ref{sec:t-perf} for a precise definition).

If the system of inequalities defining $\tstab(G)$ is totally dual integral, the graph  $G$ is called \emph{strongly $t$-perfect}. Evidently, strong $t$-perfection implies $t$-perfection. It is not known whether the converse is also true.

\begin{theorem}\label{thm:t-perfect}
For every quadrangulation $G$ of the projective plane the following assertions are equivalent:
\begin{enumerate}[\rm(a)]
\item $G$ is $t$-perfect \label{item:t-perfect}
\item $G$ is strongly  $t$-perfect \label{item:strongly_t-perfect}
\item $G$ is bipartite \label{item:bipartite}
\end{enumerate}
\end{theorem}

See Section~\ref{sec:t-perf} for precise definitions and for the proof.

A general treatment on $t$-perfect graphs may be found in Gr\"otschel, Lov\'asz and Schrijver~\cite[Ch.~9.1]{GLS88} as well as in Schrijver~\cite[Ch.~68]{LexBible}. We showed that triangulations of the projective plane are (strongly) $t$-perfect if and only if they are perfect and do not contain the complete graph $K_4$~\cite{triang17}. Bruhn and Benchetrit analysed $t$-perfection of triangulations of the sphere~\cite{Ben_Bru15}.
Boulala and Uhry~\cite{BouUhr79} established the $t$-perfection of series-parallel graphs. Ge\-rards~\cite{Gerards89} extended this to graphs that do not contain an 
\emph{odd-$K_4$} as a subgraph (an odd-$K_4$ is a subdivision of $K_4$ in which every triangle becomes an odd circuit).
Ge\-rards and Shepherd~\cite{GS98} characterised the graphs with all subgraphs $t$-perfect,
while Barahona and Mahjoub~\cite{BM94} described the $t$-imperfect subdivisions of $K_4$. 
%Further, claw-free graphs are characterised by forbidden $t$-minors (Bruhn and Stein \cite{tperfect}).
Bruhn and Fuchs~\cite{Bru_Fu15} characterised $t$-perfection of $P_5$-free graphs by forbidden $t$-minors.

\section{Quadrangulations} \label{sec:quadrangulations}
 All the graphs mentioned here are finite and simple. We follow the notation of Diestel~\cite{Diestel}.
We begin by recalling several useful definitions related to surface-embedded graphs.
For further background on topological graph theory, we refer the reader to Gross and Tucker~\cite{Gross_Tucker} or 
Mohar and  Thomassen~\cite{Mohar_Thomassen}. 

%The (real) \emph{projective plane} is defined as the space of lines through the origin in $\R^3$. 

An \emph{embedding} of a simple graph $G$ on a surface is a continuous one-to-one function from a topological representation of $G$ into the surface. 
%where $G$ is assumed to have the natural topology as a $1$-dimensional CW-complex. 
For our purpose, it is convenient to abuse the terminology by referring to the image of $G$ as the  \emph{ the graph $G$}. The \emph{faces} of an embedding are the connected components of the complement of $G$. 
An embedding $G$ is \emph{even} if all faces are bounded by an even circuit. A \emph{quadrangulation} is an embedding where each face is bounded by a circuit of length $4$. 
%A \emph{quadrangulation} is a closed-cell embedding where each face is bounded by a $4$-cycle. 
%A \emph{quadrangulation} is a cell embedding where all faces are of size $4$. Note that a quadrangulation is always a closed-cell embedding as $G$ does not contain multiple edges. 
A cycle $C$  is \emph{contractible} if $C$ separates the surface into two sets $S_C$ and $\overline{S_C}$ where $S_C$ is homeomorphic to an open disk in $\R^2$. 
 Note that for the sphere, $S_C$ and $\overline{S_C}$ are homeomorphic to an open disk. In contrast, for the plane and the projective plane, $\overline{S_C}$ is not homeomorphic to an open disk. 
For the plane and the projective plane , we call $S_C$ the \emph{interior} of $C$ and $\overline{S_C}$ the \emph{exterior} of $C$. 
Using the stereographic projection, it is easy to switch between embeddings in the sphere and the plane. In order to have an interior and an exterior of a contractible cycle, we will concentrate on quadrangulations of the plane (and the projective plane). 
Note that by the Jordan curve theorem, 
\begin{equation} \label{eq:plane->contractible_cycles}
\text{all cycles in the plane are contractible.}
\end{equation}
A cycle in a non-bipartite quadrangulation of the projective plane is contractible if and only if it has even length (see e.g.~\cite[Lemma 3.1]{Kai_Steh15}). 
As every non-bipartite even embedding is a subgraph of a non-bipartite quadrangulation, one can easily generalise this result.
\begin{observation} \label{obs:cycles_in_pp}
A cycle in a non-bipartite even embedding in the projective plane is contractible if and only it has even length. 
\end{observation}
An embedding is a \emph{$2$-cell embedding} if each face is homeomorphic to an open disk.
It is well-known that embeddings of  $2$-connected graphs in the plane are $2$-cell embeddings. 
A non-bipartite quadrangulation of the projective plane contains a non-contractible cycle; see Observation~\ref{obs:cycles_in_pp}. The complement of this cycle in the projective plane is homeomorphic to  an open disk. Thus, we observe:
\begin{observation} \label{obs:closed-cell-embedding}
Every quadrangulation of the plane and every non-bipartite quadrangulation of the projective plane is a $2$-cell embedding. 
\end{observation}
This observation makes sure that we can apply Euler's formula to all the considered quadrangulations. 
A simple graph cannot contain a $4$-circuit that is not a $4$-cycle. Thus, note that every face of a quadrangulation is bounded by a cycle. 

It is easy to see that 
\begin{equation} \label{eq:quad_plane->bipartite}
 \text{all quadrangulations of the plane are bipartite.}
\end{equation}

We first take a closer look at deletions of degree-$2$ vertices in graphs that are not the $4$-cycle $C_4$.

\begin{observation} \label{obs:deg2_vertex}
Let $G\not= C_4$ be a quadrangulation of the plane or the projective plane 
that contains a vertex $v$ of degree $2$. Then, $G-v$ is again a quadrangulation. 
\end{observation}

\begin{proof}
Let  $u$ and $u'$ be the two neighbours of $v$. Then, there are distinct vertices $s,t$ such that the cycles $(u,v,u',s,u) $ and $(u,v,u',t,u)$ are bounding a face. 
Thus, $(u,s,u',t,u)$ is a contractible $4$-cycle whose interior contains only $v$ and $G-v$ is again a quadrangulation.  
\end{proof}

We now take a closer look at $t$-contractions.

\begin{lemma} \label{lem:quad_stays_quad}
Let $G$ be a quadrangulation of the plane or a non-bipartite quadrangulation of the projective plane. Let $G'$ be obtained from $G$ by a $t$-contraction at $v$. If 
 $v$ is not a vertex of a contractible $4$-cycle with some vertices in its interior, then $G'$ is again a quadrangu\-lation.
\end{lemma}

\begin{proof} Let $G''$ be obtained from $G$ by the operation that identifies $v$ with all its neighbours but does not delete multiple edges. This operation leaves every cycle not containing $v$ untouched, transforms every other cycle $C$ into a cycle of length $|C|-2$, and creates no new cycles. Therefore, all cycles bounding faces of $G''$ are of size $4$ or $2$. The graphs $G'$ and $G''$ differ only in the property that $G''$ has some double edges. These double edges form $2$-cycles that arise from $4$-cycles containing $v$. As all these $4$-cycles are contractible (see~\eqref{eq:plane->contractible_cycles} and  Observation~\ref{obs:cycles_in_pp}) with no vertex in their interior, the $2$-cycles are also contractible and contain no vertex in its interior. Deletion of all double edges now gives $G'$ --- an embedded graph where all faces are of size $4$.

\end{proof}

 Lemma~\ref{lem:quad_stays_quad} enables us to prove the following statement that directly implies Theorem~\ref{thm:plane_C4_irreducible}. 

\begin{lemma}\label{lem:plane_C4_irreducible}
Let $G$ be a quadrangulation of the plane. Then, there is a sequence of
\begin{itemize}
\item $t$-contractions at degree-$3$ vertices that are only contained in $4$-cycles whose interior does not contain a vertex. 
\item deletions of degree-$2$ vertices 
\end{itemize}
that transforms $G$ into a $4$-cycle. During the whole process, the graph remains a quadrangulation. 
\end{lemma}

\begin{proof}
Let $\mathcal{C}$ be the set of all contractible $4$-cycles whose interior contains some vertices of $G$. Note that $\mathcal{C}$ contains the $4$-cycle bounding the outer face unless $G=C_4$.

Let $C \in \mathcal{C}$ be a contractible $4$-cycle whose interior does not contain another element of $\mathcal{C}$. 
We will first see that the interior of $C$ contains a vertex of degree $2$ or $3$: Deletion of all vertices in the exterior of $C$ gives a quadrangulation $G'$ of the plane. 
As $G$ is connected, one of the vertices in $C$ must have a neighbour in the interior of $C$ and thus must have  degree at least $3$. Euler's formula  now implies that $
\sum_{v \in V(G')} \deg(v)=2|E(G')| \leq 4|V(G')| -8.
$ As no vertex in $G'$ has degree $0$ or $1$, there must be a vertex of degree $2$ or $3$ in $V(G')-V(C)$. This vertex has the same degree in $G$ and is contained in the interior of $C$.

We use deletions of degree-$2$ vertices and $t$-contractions at degree-$3$ vertices in the interior of the smallest cycle of $\mathcal{C} $ to successively get rid of all vertices in the interior of $4$-cycles. By Observation~\ref{obs:deg2_vertex} and Lemma~\ref{lem:quad_stays_quad}, the obtained graphs are  quadrangulations.
Now, suppose that no more $t$-contraction at a degree-$3$ vertex and no more deletion of a degree-$2$ vertex is possible. Assume that the obtained graph is not a $4$-cycle. Then, there is a cycle $C'\in \mathcal{C}$  whose interior does not contain another cycle of $\mathcal{C}$. 
As we have seen above, $C' \in \mathcal{C}$  contains a vertex $v$ of degree $3$. Since no $t$-contraction can be applied to $v$, the vertex $v$ has two adjacent neighbours.
This contradicts~\eqref{eq:quad_plane->bipartite}. 
\end{proof}

In the rest of the paper, we will consider the projective plane.

A quadrangulation of the projective plane is~\emph{nice} if no vertex is contained in the interior of a contractible $4$-cycle.
 
\begin{lemma}\label{lem:make_quad_nice}
Let $G$ be a non-bipartite quadrangulation of the projective plane. Then, there is a sequence of $t$-contractions and deletions of vertices of degree $2$ that transforms $G$ into a nice quadrangulation. During the whole process, the graph remains a quadrangulation.
\end{lemma}

\begin{proof}
Let $C$ be a contractible $4$-cycle whose interior contains at least one vertex. Delete all vertices that are contained in the exterior of $C$. The obtained graph is a quadrangulation of the plane. By Lemma~\ref{lem:plane_C4_irreducible}, there is a sequence of $t$-contractions (as described in Lemma~\ref{lem:quad_stays_quad}) and deletions of degree-$2$ vertices that eliminates all vertices in the interior of $C$. 
With this method, it is possible to transform $G$ into a nice quadrangulation. 

%Proceed as in the proof of Theorem~\ref{thm:plane_C4_irreducible}, starting with the second sentence and ending with~\eqref{eq:v_in_triangle}.
%
%As the $3$-cycle must be non-contractible (Lemma~\ref{lem:cycles_in_pp}), the two neighbours of $v$ are two opposite vertices of $C$, say $C=v_1, v_2, v_3, v_4, v_1$ and $v_1v, v_3v \in E(G)$. Now, $(v_1,v_2,v_3,v)$ and $(v_1,v_4,v_3,v)$ are contractible $4$-cycles in $C$. As $v$ is not of degree $2$, at least one of these cycles, say $(v_1,v_2,v_3,v)$,  contains a neighbour $w$ of $v$  in its interior. By assumption, $w$ is neither of degree $2$ nor can we apply a $t$-contraction at $w$. Again, $w$ must be in a non-contractible triangle. Therefore, $w$ is adjacent to $v_1$ and $v_3$ and $G$ contains the contractible $3$-cycle $w,v,v_1,w$. This contradicts Lemma~\ref{lem:cycles_in_pp}.

\end{proof}

Similar as in the proof of Theorem~\ref{thm:plane_C4_irreducible}, Euler's formula implies that a non-bipartite quadrangulation of the projective plane contains a vertex of degree $2$ or $3$.  As no nice quadrangulation has a degree-$2$ vertex (see Observation~\ref{obs:deg2_vertex}), we deduce: 

\begin{observation}\label{obs:nice_quad_min_degree}
Every nice non-bipartite quadrangulation of the projective plane has minimal degree $3$.
\end{observation}

 In an even embedding of an odd wheel $W$, every odd cycle must be non-contractible (see Observation~\ref{obs:cycles_in_pp}). Thus, it is easy to see that there is only one way (up to topological isomorphy) to embed an odd wheel in the projective plane. (This can easily be deduced from~\cite{MoRoVi96} --- a paper dealing with embeddings of planar graphs in the projective plane.) The embedding is illustrated in Figure~\ref{fig:pp_wheels}. Noting that this embedding is a quadrangulation, we observe:
 
\begin{observation} \label{obs:odd_wheel_nice_quad}
 Let $G$ be a quadrangulation of the projective plane that contains an odd wheel $W$. If $G$ is nice, then $G$ equals $W$. 
\end{observation}

Note that every graph containing an odd wheel also contains an induced odd wheel. Now, we consider even wheels.

\begin{lemma} \label{lem:even_wheel_no_embedding}
Even wheels $W_{2k}$ for $k\geq 2$ do not have an even embedding in the projective plane.
\end{lemma}
The statement follows directly from~\cite{MoRoVi96}.  We nevertheless give an elementary proof of the lemma. 

\begin{proof}
First assume that the $4$-wheel $W_4$ has an even embedding. As all triangles of $W_4-{w_3w_4}$ must be non-contractible by Observation~\ref{obs:cycles_in_pp}, it is easy to see that the graph must be embedded as in Figure~\ref{fig:pp_evenwheels}. Since the insertion of $w_3w_4$ will create an odd face, $W_4$ is not evenly embeddable.

  Now assume that $W_{2k}$ for $k \geq 3$ is evenly embedded. Delete the edges $vw_i$ for $i=5, \ldots, 2k$ and note that $w_5, \ldots, w_{2k}$ are now of degree $2$, ie the path $P=(w_4, w_5, \ldots, w_{2k}, w_1)$ bounds two faces or one face from two sides.
Deletion of the edges $vw_i$ preserve the even embedding:
Deletion of an edge bounding two faces $F_1, F_2$  merges the faces into a new face of size $|F_1|+|F_2|-2$. Deletion of an edge bounding a face $F$ from two sides leads to a new face of size $|F|-2$. In both cases, all other faces are left untouched.

Next, replace the odd path $P$ by the edge $w_4w_1$. The two faces $F_3, F_4$ adjacent to $P$ are transferred into two new faces of size $|F_3|- (2k-3)+1$ and $|F_4|- (2k-3)+1$. This yields an even embedding of $W_4$ which is a contradiction.
\end{proof}

\begin{figure}[bht]
\begin{center}
\begin{tikzpicture}[scale = .55]
\begin{scope}
\draw[dotted] (0,0) circle (2cm);
\def\krad{1cm}
\def\brad{2cm}
\def\angle{360/4}
\def\sangle{360/8}
\node[hvertex] (c) at (0,0){};
\node[novertex] (uc) at (0,-0.6){$v$};
\foreach \i in {1,...,4}{
  \node[hvertex] (v\i) at (270-\angle/2+\i*\angle:\krad){};
  %\node[novertex] (u\i) at (270-\angle/2+\i*\angle:\krad){$w_{\i}$};
  \draw[hedge] (c) -- (v\i);
  }
\foreach \i in {0,2,3,4,6,7}{
  \node[novertex] (w\i) at (270-\angle/2+\i*\sangle:\brad){};
}
\draw[hedge] (v4) -- (w0);
\draw[hedge] (v2) -- (w4);
\draw[hedge] (v1) -- (w2);
\draw[hedge] (v3) -- (w6);
\draw[hedge] (v2) -- (w3);
\draw[hedge] (v3) -- (w7);

\node[novertex] (u4) at (1.4,-0.7){$w_4$};
\node[novertex] (u3) at (-1.4,-0.7){$w_3$};
\node[novertex] (u2) at (0.5,1.2){$w_2$};
\node[novertex] (u1) at (-0.5,1.2){$w_1$};
\end{scope}
\end{tikzpicture}
\end{center}

\caption{The only even embedding of $W_4 -  w_3w_4 $ in the projective plane. Opposite points on the dotted cycle are identified.}
\label{fig:pp_evenwheels}
\end{figure}
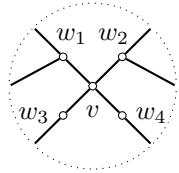

Note that  a $t$-contraction at a vertex $v$ is only allowed if its neighbourhood is stable, that is, if $v$ is not contained in a triangle. The next lemma characterises the quadrangulations to which no $t$-contraction can be applied.

\begin{lemma} \label{lem:irred_oddwheel}
Let $G$ be a non-bipartite nice quadrangulation of the projective plane where each vertex is contained in a triangle. 
Then $G$ is an odd wheel. 
\end{lemma}

\begin{proof} 
By Observation~\ref{obs:nice_quad_min_degree}, there is a vertex $v$ of degree $3$ in $G$. Let  $\{x_1,x_2,x_3\}$ be its neighbourhood and let $x_1$,$x_2$ and $v$ form a triangle. 

Recall that each two triangles are non-contractible (see~Observation~\ref{obs:cycles_in_pp}). Consequently each two triangles intersect.  
As $x_3$ is contained in a triangle intersecting the triangle $(v,x_1,x_2)$ and as $v$ has no further neighbour, we can suppose without loss of generality that $x_3$ is adjacent to $x_1$. 
The graph induced by the two triangles $(v,x_1,x_2)$ and $(x_1,v,x_3)$ is not a quadrangulation. Further, addition of the edge $x_2x_3$ yields a $K_4$. 
By Observation~\ref{obs:odd_wheel_nice_quad}, $G$ then equals the odd wheel $W_3=K_4$.

Otherwise, the graph contains a further vertex and this vertex is contained in a further triangle $T$.
Since the vertex $v$ has degree $3$, it is not contained in $T$. 
If further $x_1 \notin V(T)$, then the vertices $x_2$ and $x_3$ must be contained in $T$. But then $x_2x_3\in E(G)$ and, as above, $v$, $x_1$, $x_2$ and $x_3$ form a $K_4$. 
Therefore, $x_1$ is contained in $T$ and consequently in every triangle of $G$. Since every vertex is contained in a triangle, $x_1$ must be adjacent to all vertices of $ G-x_1$. 
%As $G$ is not almost bipartite (Lemma~\ref{lem:quadr_not_alm_bip}), $G'$ contains an odd cycle. This cycle forms an odd wheel together with $x_1$.
As $|E(G)|=2|V(G)|-2$ by Euler's formula, the graph $G-x_1$ has $2|V(G)|-2-(|V(G)|-1)=|V(G)|-1=|V(G-x_1)|$ many edges. By Observation~\ref{obs:nice_quad_min_degree}, no vertex in $G$ has degree smaller than $3$. Consequently, no vertex in $G-x_1$ has degree smaller than $2$. Thus, $G-x_1$ is a cycle and $G$ is a wheel. 
By Lemma~\ref{lem:even_wheel_no_embedding}, $G$ is an odd wheel. 
\end{proof}

Finally, we can prove our second main result:
\begin{proof}[Proof of Theorem~\ref{thm:pp_odd_wheels_irreducible}]
Transform $G$ into a nice quadrangulation (Lemma~\ref{lem:make_quad_nice}). Now,
consecutively apply $t$-contractions (as described in Lemma~\ref{lem:quad_stays_quad}) as long as possible.
In each step, the obtained graph is a quadrangulation. By Lemma~\ref{lem:make_quad_nice} we can assume that the quadrangulation is nice. If no more $t$-contraction can be applied, then every vertex is contained in a triangle. By Lemma~\ref{lem:irred_oddwheel}, the obtained quadrangulation is an odd wheel.
\end{proof}

\section{(Strong) $t$-perfection}\label{sec:t-perf}

The \emph{stable set polytope} $\ssp(G)\subseteq\mathbb R^{V}$ of a graph $G=(V,E)$ 
is defined as the convex 
hull of the characteristic vectors of stable, ie independent, subsets of $V$.
The characteristic vector of a subset $S$ of the set $V$ is the vector $\charf{S}\in \{0,1\}^{V}$ with $\charf{S}(v)=1$ if $v\in S$ and $0$ otherwise. 
We define a second polytope $\tstab(G)\subseteq\mathbb R^V$ for $G$, given by
\begin{eqnarray} \label{inequ}
&&x\geq 0,\notag\\
&&x_u+x_v\leq 1\text{ for every edge }uv\in E,\\ 
&&\sum_{v\in V(C)}x_v\leq \left\lfloor\frac{ |C|}{2}\right\rfloor\text{ for every induced odd cycle }C
\text{ in }G.\notag
\end{eqnarray}
These inequalities are respectively known as non-negativity, edge and
odd-cycle inequalities. Clearly, $\ssp(G)\subseteq \tstab(G)$.
The graph $G$ is called \emph{$t$-perfect} if $\ssp(G)$ and
$\tstab(G)$ coincide. Equivalently, $G$ is $t$-perfect if and only if
$\tstab(G)$ is an integral polytope, ie if all its vertices
are integral vectors. 
The graph $G$ is called \emph{strongly $t$-perfect} if the system \eqref{inequ} of inequalities is totally dual integral. That is, if for each weight vector $w \in \Z^{V}$, the linear program of maximizing $w^Tx$ over~\eqref{inequ} has an integer optimum dual solution. 
This property implies that $\tstab(G)$ is integral. Therefore, 
strong $t$-perfection implies $t$-perfection.
It is an open question whether every $t$-perfect graph is strongly $t$-perfect. 
The question is briefly discussed in Schrijver~\cite[Vol. B, Ch. 68]{LexBible}.

It is easy to see that all bipartite graphs are (strongly) $t$-perfect (see eg Schrijver~\cite[Ch.~68]{LexBible}) and that vertex deletion preserves (strong) $t$-perfection.
Another operation that keeps (strong) $t$-perfection  (see eg~\cite[Vol.~B,~Ch.~68.4]{LexBible}) was found by Gerards and Shepherd~\cite{GS98}: the $t$-contraction. 

%Vertex deletion and $t$-contraction preserve (strong) $t$-perfection (see eg~\cite[Vol.~B,~Ch.~68.4]{LexBible}). 
%Any graph that is  obtained from $G$ by a sequence of vertex deletions 
%and $t$-contractions is a \emph{$t$-minor of $G$}.
%Let us point out that any $t$-minor of a (strongly) $t$-perfect graph is again (strongly) $t$-perfect.

Odd wheels $W_{2k+1}$ for $k \geq 1$ are not (strongly) $t$-perfect. Indeed, the vector $(1 \slash 3, \ldots, 1 \slash 3)$  is contained in $\tstab(W_{2k+1})$ but not in $\ssp(W_{2k+1})$. 
%Furthermore, every proper $t$-minor of an odd wheel is  $t$-perfect.

With this knowledge, the proof of Theorem~\ref{thm:t-perfect} follows directly from Theorem~\ref{thm:pp_odd_wheels_irreducible}.

\begin{proof}  [Proof of Theorem~\ref{thm:t-perfect}]
If $G$ is bipartite, the $G$ is (strongly) $t$-perfect. 

Let $G$ be non-bipartite. Then, there is a sequence of $t$-contractions and deletions of vertices that transforms  $G$ into an odd wheel (Theorem~\ref{thm:pp_odd_wheels_irreducible}). As odd wheels are not (strongly) $t$-perfect and as vertex deletion and $t$-contraction preserve (strong) $t$-perfection, $G$ is not (strongly) $t$-perfect.
\end{proof}

\bibliographystyle{abbrv}
\bibliography{bibquad}

\vfill
\noindent
Elke Fuchs
{\tt <elke.fuchs@uni-ulm.de>}\\
Laura Gellert
{\tt <laura.gellert@uni-ulm.de>}\\
Institut f\"ur Optimierung und Operations Research\\
Universit\"at Ulm, Ulm\\
Germany\\

\end{document}